\documentclass[oneside,english]{amsart}
\usepackage[T1]{fontenc}
\usepackage[latin9]{inputenc}
\usepackage{textcomp}
\usepackage{mathtools}
\usepackage{amsbsy}
\usepackage{amstext}
\usepackage{amsthm}
\usepackage{amssymb}
\usepackage{esint}

\makeatletter
\numberwithin{equation}{section}
\numberwithin{figure}{section}
\theoremstyle{plain}
\newtheorem{thm}{\protect\theoremname}[section]
\theoremstyle{remark}
\newtheorem{rem}[thm]{\protect\remarkname}
\theoremstyle{plain}
\newtheorem{cor}[thm]{\protect\corollaryname}
\theoremstyle{plain}
\newtheorem{conjecture}[thm]{\protect\conjecturename}
\theoremstyle{plain}
\newtheorem{fact}[thm]{\protect\factname}
\theoremstyle{plain}
\newtheorem{lem}[thm]{\protect\lemmaname}

\makeatother

\usepackage{babel}
\providecommand{\conjecturename}{Conjecture}
\providecommand{\corollaryname}{Corollary}
\providecommand{\factname}{Fact}
\providecommand{\lemmaname}{Lemma}
\providecommand{\remarkname}{Remark}
\providecommand{\theoremname}{Theorem}

\begin{document}
\title[positive ground state solutions for the Choquard type equation]{The existence of positive ground state solutions for the Choquard
type equation on groups of polynomial growth}
\author{Ruowei Li}
\address{Ruowei Li: School of Mathematical Sciences, Fudan University, Shanghai
200433, People\textquoteright s Republic of China; Shanghai Center
for Mathematical Sciences, Fudan University, Shanghai 200433, People\textquoteright s
Republic of China}
\email{rwli19@fudan.edu.cn}
\begin{abstract}
In this paper, let $G$ be a Cayley graph of a discrete group of polynomial
growth with homogeneous dimension $N\geq3$. We study the Choquard
type equation on $G$:
\begin{equation}
\Delta u+(R_{\alpha}\ast\mid u\mid^{p})\mid u\mid^{p-2}u=0,\label{eq: choquard equation}
\end{equation}
where $\alpha\in(0,N)$, $p>\frac{N+\alpha}{N-2}$ and $R_{\alpha}$
stands for the Green\textquoteright s function of the discrete fractional
Laplace operator, which has same asymptotics as the Riesz potential.
We prove the discrete Hardy-Littlewood-Sobolev inequality on such
Cayley graphs, and by the discrete Concentration-Compactness principle
we prove the existence of extremal functions for the corresponding
Sobolev type inequalities in supercritical cases, which yields a positive
ground state solution of (\ref{eq: choquard equation}). Moreover,
we obtain positive ground state solutions of Choquard type equations
with $p$-Laplace, biharmonic and $p$-biharmonic operators etc.
\end{abstract}

\maketitle

\section{Introduction}

The nonlinear Choquard equation 
\begin{equation}
\Delta u+u=(I_{\alpha}\ast\mid u\mid^{p})\mid u\mid^{p-2}u,\;\text{in \ensuremath{\mathbb{R}^{N}}}\label{eq: classical choquard equation}
\end{equation}
where $I_{\alpha}(x)=\frac{\Gamma(\frac{N-\alpha}{2})}{\Gamma(\frac{\alpha}{2})\pi^{N/2}2^{\alpha}\mid x\mid^{N-\alpha}}$
is the Riesz potential, has been extensively studied in the literature.
In the physical case $N=3,p=2$ and $\alpha=2,$ in 1954 the problem
appeared in a work by S. I. Pekar describing the quantum mechanics
of a polaron at rest \cite{P54}. In 1976 P. Choquard used (\ref{eq: classical choquard equation})
to describe an electron trapped in its own hole, in a certain approximation
to Hartree--Fock theory of one component plasma \cite{L76}. In 1996
R. Penrose proposed (\ref{eq: classical choquard equation}) as a
model of self-gravitating matter, in a programme in which quantum
state reduction is understood as a gravitational phenomenon \cite{MPT98}.
The equations of type (\ref{eq: classical choquard equation}) are
usually called the nonlinear Schrödinger--Newton equation. If $u$
solves (\ref{eq: classical choquard equation}) then the function
$\psi$ defined by $\psi(t,x)=e^{it}u(x)$ is a solitary wave of the
focusing time-dependent Hartree equation
\[
i\psi_{t}=-\Delta\psi-(I_{\alpha}\ast\mid\psi\mid^{p})\mid\psi\mid^{p-2}\psi,\;\text{in \ensuremath{\mathbb{R}_{+}\times\mathbb{R}^{N}}}.
\]
In this context (\ref{eq: classical choquard equation}) is also known
as the stationary nonlinear Hartree equation.

The existence of solutions was proved for $N=3,p=2$ and $\alpha=2,$
with the symmetrization method, the Strauss lemma \cite{S77}, and
the ordinary differential equation techniques by E. H. Lieb, P. L.
Lions and G. Menzala \cite{L76,L1,L80}. Li Ma and Lin Zhao classified
all positive solutions by the method of moving plane in an integral
form \cite{MZ10}. For higher dimensions, P. L. Lions \cite{L1,L2,L3,L4}
established the Concentration-Compactness method, which provided a
new idea for proving the existence result. The general idea is as
follows. Since the problem (\ref{eq: classical choquard equation})
has a variational structure, take a maximizing sequence $\{u_{n}\}$
for the variational problem
\[
\text{sup\ensuremath{\left\{ \underset{\mathbb{R}^{N}}{\int}(I_{\alpha}\ast\mid u\mid^{p})\mid u\mid^{p}:\lVert u\rVert_{H^{1}}=1\right\} }}
\]
and regard $\left\{ \left(|\nabla u_{n}|^{2}+u_{n}^{2}\right)\mathrm{dx}\right\} $
as a sequence of probability measures. He proved in \cite{L1,L3}
that there are three cases of the limit of the sequence: compactness,
vanishing and dichotomy. Vanishing and dichotomy are ruled out by
the rescaling trick and the subadditivity inequality. Therefore, the
nontrivial limit function, which is called the extremal function,
exists by the compactness, and after scaling it satisfies (\ref{eq: classical choquard equation}).
By the Concentration-Compactness method of P. L. Lions \cite[Lemma I.1]{L1}
(see also \cite[Lemma 1.21]{W96}, \cite{GdEYZ20}), V. Moroz and
J. Van Schaftingen in \cite{MV13} proved the optimal parameter range
of the existence of a positive ground state solution, i.e. $\frac{N+\alpha}{N}<p<\frac{N+\alpha}{N-2}$,
and the range is sharp. They established regularity, constant sign,
radially symmetric, monotone decaying, and decay asymptotics of the
ground state solutions. They also studied the Choquard equations with
a general class of nonhomogeneous nonlinearities \cite{MV15}, and
with a potential term $V$ \cite{MV15-2}. Multiple solutions results
were obtained in \cite{CCS12}. There are intensive studies of the
Choquard equations, see survey papers \cite{MV17,SM20} and \cite{TM99,GY18,MV15-3,DSS15,GV16,AY14,GY17,ANY16,M80,SGY16,MV13-2}
references therein.

In recent years, people paid attention to the analysis on discrete
spaces, especially for the nonlinear equations, see for example \cite{GLY16,GLY17,HLY20,CMW16,LZY20,GJ18,HSZ20,ZZ18,HL21,HX21,HLW22,HLM22}.
As far as we know, there is no existence results for the nonlinear
Choquard equation on graphs as the continuous setting. In this article,
we prove the existence results on graphs by discrete Concentration-Compactness
principle and generalize our previous results on the existence of
solutions for $p$-Laplace equations and $p$-biharmonic equations
\cite{HL21,HLM22} to the Choquard type equations.

Let $(G,S)$ be a Cayley graph of a group $G$ with a finite symmetric
generating set $S$, i.e. $S=S^{-1}.$ There is a natural metric on
$(G,S)$ called the word metric, denoted by $d^{S}$. Let $B_{p}^{S}(n):={\left\{ x\in G\mid d^{S}(p,x)\leq n\right\} }$
denote the closed ball of radius $n$ centered at $p\in G$ and denote
$\mid B_{p}^{S}(n)\mid:=\sharp B_{p}^{S}(n)$ as the volume (i.e.
cardinality) of the set $B_{p}^{S}(n)$. When $e$ is the unit element
of $G$, the volume $\beta_{S}(n):=\mid B_{e}^{S}(n)\mid$ of $B_{e}^{S}(n)$
is called the growth function of the group, see \cite{M68-1,M68-2,W68,G90,G97,S55,G14-2}.
A group $G$ is called of polynomial growth, or of polynomial volume
growth, if $\beta_{S}(n)\leq Cn^{A}$, for any $n\geq1$ and some
$A>0$, which is independent of the choice of the generating set $S$
since the metrics $d^{S}$ and $d^{S_{1}}$ are bi-Lipschitz equivalent
for different finite generating sets $S$ and $S_{1}$. By Gromov's
theorem and Bass' volume growth estimate of nilpotent groups \cite{B72},
for any group $G$ of polynomial growth there are constants $C_{1}(S)$,
$C_{2}(S)$ depending on $S$ and $N\in\mathbb{N}$ such that for
any $n\geq1$, 
\[
C_{1}(S)n^{N}\leq\beta_{S}(n)\leq C_{2}(S)n^{N},
\]
where the integer $N$ is called the homogeneous dimension or the
growth degree of $G$. Since $N$ is a sort of dimensional constant
of $G$, we always omit the dependence of $N$ in various constants. 

In this paper, we consider the Cayley graph $(G,S)$ of a group of
polynomial growth with the homogeneous dimension $N\geq3$. In particular,
$\mathbb{Z}^{N}$ is a Cayley graph of a free abelian group. We denote
by $\ell^{p}(G)$ the $\ell^{p}$-summable functions on $G$ and by
$D_{0}^{k,p}(G)$ $(k=1,2)$ the completion of finitely supported
functions in the $D^{k,p}$ norm, where
\[
\lVert u\rVert_{D^{1,p}(G)}\coloneqq\lVert\lvert\nabla u\rvert_{p}^{p}\rVert_{\ell^{1}(G)}^{1/p}=\left(\sum\limits _{x\in G}\sum\limits _{y\sim x}\lvert\nabla_{xy}u\rvert^{p}\right)^{1/p},\;\nabla_{xy}u\coloneqq u(y)-u(x),
\]
\[
\lVert u\rVert_{D^{2,p}(G)}\coloneqq\lVert\Delta u\rVert_{\ell^{p}(G)},\;\Delta u(x)\coloneqq\sum\limits _{y\sim x}\nabla_{xy}u,
\]
see Section 2 for details.

By a standard trick and the isoperimetric estimate \cite[Theorem 4.18]{W00},
the discrete Sobolev inequality (\ref{eq:discrete sobo}) holds on
$G$, see \cite[Theorem 3.6]{HM15}, 

\begin{equation}
\lVert u\rVert_{\ell^{q}}\leq C_{p}\lVert u\rVert_{D^{1,p}},\;\forall u\in D_{0}^{1,p}(G),\label{eq:discrete sobo}
\end{equation}
where $N\geq2,1\leq p<N,q=p^{\ast}\coloneqq\dfrac{Np}{N-p}$. Since
$\ell^{p}(G)$ embeds into $\ell^{q}(G)$ for any $q>p$, see \cite[Lemma 2.1]{HLY15},
one verifies that the discrete Sobolev inequality (\ref{eq:discrete sobo})
hold when $q\geq p^{\ast}$. Recalling the continuous setting, it
is called the subcritical for $q<p^{\ast}$, critical for $q=p^{\ast}$
and supercritical for $q>p^{\ast}$ for the Sobolev inequality. Therefore,
(\ref{eq:discrete sobo}) hold in both critical and supercritical
cases on $G$.

The discrete Hardy-Littlewood-Sobolev (HLS for abbreviation) inequality
with the Riesz potential $I_{\alpha}$ has been proved on $\mathbb{Z}^{N}$
\cite{HLY15},
\begin{equation}
\sum_{\substack{x,y\in\mathbb{Z}^{N}\\
x\neq y
}
}\frac{f(x)g(y)}{\mid x-y\mid^{N-\alpha}}\leq C_{r,s,\alpha}\Vert f\Vert_{\ell^{r}}\Vert g\Vert_{\ell^{s}},\:\forall f\in\ell^{r}(\mathbb{Z}^{N}),g\in\ell^{s}(\mathbb{Z}^{N}),\label{eq: discrete hls}
\end{equation}
where $r,s>1,$ $0<\alpha<N$, $\frac{1}{r}+\frac{1}{s}+\frac{N-\alpha}{N}=2$.
Since the Riesz potential $I_{\alpha}$ is exactly the Green's function
of the fractional Laplace on $\mathbb{R}^{N}$, $\alpha\in(0,N)$,
it is natural to consider the discrete fractional Laplace $(-\Delta)^{\frac{\alpha}{2}}$
and its Green's function $R_{\alpha}$ on $G$ (see \cite{MCRNN17}
for $\mathbb{Z}^{N}$). The heat kernel $k_{t}(x,y)$ of the Laplace
operator $\Delta$ on $G$ has the Gaussian heat kernel bounds \cite[Theorem 6.19]{B17}.
By the method of subordination and Bochner\textquoteright s functional
calculus \cite{SSV12,BBKRSV80}, the fractional Laplace operator on
$G$ is defined as
\[
(-\Delta)^{\frac{\alpha}{2}}u\coloneqq\sideset{\frac{1}{\mid\varGamma(-\frac{\alpha}{2})\mid}}{_{0}^{\infty}}\int\left(e^{t\Delta}u-u\right)t^{-1-\frac{\alpha}{2}}\text{d}t,
\]
where $e^{t\Delta}$ is the semigroup of $\Delta$, see \cite{KL12}.
And the Green's function $R_{\alpha}$ of the fractional Laplace on
$G$ is
\[
R_{\alpha}(x,y)=\frac{1}{\Gamma(\frac{\alpha}{2})}\int_{0}^{\infty}k_{t}(x,y)t^{-1+\frac{\alpha}{2}}\text{d\ensuremath{t}},\;x,y\in G,
\]
which has the asymptotic relation $R_{\alpha}(x,y)\simeq\left(d^{S}(x,y)\right)^{\alpha-N}$.
By the Young\textquoteright s inequality for weak type spaces \cite[Theorem 1.4.25.]{G14},
we get the discrete HLS inequality on $G$, see Theorem \ref{thm: hls on G},
\begin{equation}
\sum_{\substack{x,y\in G\\
x\neq y
}
}R_{\alpha}(x,y)f(x)g(y)\leq C_{r,s,\alpha}\Vert f\Vert_{\ell^{r}}\Vert g\Vert_{\ell^{s}},\:\forall f\in\ell^{r}(G),g\in\ell^{s}(G),\label{eq: hls1}
\end{equation}
where $r,s>1,$ $0<\alpha<N$, $\frac{1}{r}+\frac{1}{s}+\frac{N-\alpha}{N}=2$,
and an equivalent form of the HLS inequality is

\begin{equation}
\lVert R_{\alpha}\ast f\rVert_{\ell^{t}}\leq C_{r,\alpha}\lVert f\rVert_{\ell^{r}},\:\forall f\in\ell^{r}(G),\label{eq:hls2}
\end{equation}
where $1<r<\frac{N}{\alpha},0<\alpha<N,t=\frac{Nr}{N-\alpha r}$.
Similarly, (\ref{eq: hls1}) and (\ref{eq:hls2}) hold in both critical
and supercritical cases, i.e. $\frac{1}{r}+\frac{1}{s}+\frac{N-\alpha}{N}\geq2$
and $t\geq\frac{Nr}{N-\alpha r}$ respectively. Hence by (\ref{eq: hls1})
and (\ref{eq:hls2}) we obtain the following Sobolev type inequality
whose Euler-Lagrange equation is (\ref{eq: choquard equation}),

\begin{align}
\sum_{G}\left(R_{\alpha}\ast\mid u\mid^{p}\right) & \mid u\mid^{p}\leq C_{p,\alpha}\Vert u\Vert_{\ell^{\frac{2Np}{N+\alpha}}}^{2p}\label{eq: main inequality}\\
 & \leq C_{p,\alpha}\tilde{C}_{\frac{2Np}{N+\alpha}}\lVert u\rVert_{D^{1,2}}^{2p},\;\forall u\in D^{1,2}(G),\nonumber 
\end{align}
where $N\geq3$, $\alpha\in(0,N)$, $p\geq\frac{N+\alpha}{N-2}$,
$C_{p,\alpha}$, $\tilde{C}_{\frac{2Np}{N+\alpha}}$ are constants
in the HLS inequality (\ref{eq: hls1}) and the Sobolev inequality
(\ref{eq:discrete sobo}) respectively. Also we call that (\ref{eq: main inequality})
holds in both critical and supercritical cases.

The optimal constant in the inequality (\ref{eq: main inequality})
is given by 
\begin{equation}
K:=\underset{\lVert u\rVert_{D^{1,2}}=1}{\text{sup}}\sum_{G}\left(R_{\alpha}\ast\mid u\mid^{p}\right)\mid u\mid^{p}.\label{eq: sup}
\end{equation}
In order to prove that the optimal constant is achieved by some $u$,
which is called the extremal function, we consider a maximizing sequence
$\{u_{n}\}\subset D^{1,2}(G)$ satisfying 
\begin{equation}
\lVert u_{n}\rVert_{D^{1,2}}=1,\;\sum_{G}\left(R_{\alpha}\ast\mid u_{n}\mid^{p}\right)\mid u_{n}\mid^{p}\to K,n\to\infty.\label{eq:max seq}
\end{equation}
We want to prove $u_{n}\to u$ strongly in $D^{1,2}(G)$, which yields
that $u$ is a maximizer. 

We prove the following main result.
\begin{thm}
\label{thm:main1}For $N\geq3$, $\alpha\in(0,N)$, $p>\frac{N+\alpha}{N-2}$,
let $\left\{ u_{n}\right\} \subset D^{1,2}(G)$ be a maximizing sequence
satisfying (\ref{eq:max seq}). Then there exists a sequence $\{x_{n}\}\subset G$
and $v\in D^{1,2}(G)$ such that the sequence after translation $\left\{ v_{n}(x):=u_{n}(x_{n}x)\right\} $
contains a convergent subsequence that converges to $v$ in $D^{1,2}(G)$.
And $v$ is a maximizer for $K$. 
\end{thm}

\begin{rem}
(1) This result implies that the best constant can be obtained in
the supercritical case $p>\frac{N+\alpha}{N-2}$.

(2) Different from the continuous setting, $u\in D^{1,2}(G)\text{\textbackslash}\left\{ 0\right\} $,
does not yield $\sum\left(R_{\alpha}\ast\mid u\mid^{p}\right)\mid u\mid^{p}\neq0$,
for example, $u(e)=1$, $u(x)=0$ for $x\neq e$, where $e$ is the
unit element of $G$. Hence the normalized condition $\lVert u_{n}\rVert_{D^{1,2}}=1$,
rather than $\sum\left(R_{\alpha}\ast\mid u_{n}\mid^{p}\right)\mid u_{n}\mid^{p}=1$,
in the variational problem helps to rule out the vanishing case, see
Lemma \ref{lem:lower bound}.
\end{rem}

We will provide two proofs for the main result. In the continuous
setting, Lions proved the existence of extremal functions by the Concentration-Compactness
principle \cite[Theorem III.3.]{L3} and a rescaling trick \cite[Theorem I.1, (17)]{L3}.
And Lieb in \cite{L5} used a compactness technique and the rearrangement
inequalities. Following Lions, the main idea of proof I is to prove
a discrete analog of Concentration-Compactness principle, see Lemma
\ref{lem:Concentration-Compact}. However, we don't know proper notion
of the rescaling trick on $G$ to exclude the vanishing case of the
limit function. Inspired by \cite{HLY15}, for the supercritical case,
we prove that the normalized maximizing sequence after translation
has a uniform positive lower bound at the unit element, see Lemma
\ref{lem:lower bound}, which excludes the vanishing case. The idea
of proof II is based on a compactness technique by Lieb \cite[Lemma 2.7]{L5}
and the nontrivial nonvanishing of the limit of the translation sequence.
\begin{cor}
\label{cor: Coro}For $N\geq3$, $\alpha\in(0,N)$, $p>\frac{N+\alpha}{N-2}$,
there is a positive ground state solution of the equation 
\begin{equation}
\Delta u(x)+(R_{\alpha}\ast u{}^{p}(x))u^{p-1}(x)=0,\;x\in G,\label{eq: choquard eq2}
\end{equation}
Equivalently, there exists a pair of positive solution $(u,v)=(u,R_{\alpha}\ast u{}^{p})$
of the following system
\begin{equation}
\begin{cases}
\begin{array}{c}
\Delta u+vu^{p-1}=0\\
\left(-\Delta\right)^{\frac{\alpha}{2}}v=u^{p}
\end{array}.\end{cases}\label{eq:R_a system}
\end{equation}
\end{cor}

\begin{rem}
(1) \label{rem:equivalent form of HLS}By the equivalent form of the
HLS inequality (\ref{eq:hls2}) we obtain another inequality 
\[
\lVert R_{\alpha}\ast\mid u\mid^{p}\lVert_{^{\frac{2N}{N-\alpha}}}\lesssim\lVert u{}^{p}\lVert_{^{\frac{2N}{N+\alpha}}}\lesssim\lVert u\rVert_{D^{1,2}}^{p},\;\forall u\in D^{1,2}(G),
\]
here we omit the constants in inequalities. Similarly, there exist
positive solutions $\bar{u}$ and $(\bar{u},\bar{v})=(\bar{u},R_{\alpha}\ast\bar{u}{}^{p})$
of the following equation and system respectively,
\begin{equation}
\Delta u+(R_{\alpha}\ast u{}^{p})^{\frac{N+\alpha}{N-\alpha}}u^{p-1}=0,\label{eq:R_a 2}
\end{equation}
\begin{equation}
\begin{cases}
\begin{array}{c}
\Delta u+v^{\frac{N+\alpha}{N-\alpha}}u^{p-1}=0\\
\left(-\Delta\right)^{\frac{\alpha}{2}}v=u^{p}
\end{array} & .\end{cases}\label{eq:R_a system 2}
\end{equation}

(2) By the discrete HLS inequality (\ref{eq: discrete hls}) on $\mathbb{Z}^{N}$,
the above existence results can be obtained on $\mathbb{Z}^{N}$ when
replacing $R_{\alpha}$ with $I_{\alpha}$ in (\ref{eq: choquard eq2}),
(\ref{eq:R_a system}), (\ref{eq:R_a 2}) and (\ref{eq:R_a system 2}).

(3) The range of the parameter for the existence of a positive ground
state solution $p>\frac{N+\alpha}{N-2}$ is different from the continuous
case $\frac{N+\alpha}{N}<p<\frac{N+\alpha}{N-2}$, thanks to the discrete
nature, that is, $\ell^{p}(G)$ embeds into $\ell^{q}(G)$ for any
$q>p$. Hence we don't need the control of $L^{2}$ norm in the right
hand side of (\ref{eq: main inequality}), and we can remove the $u$
term in (\ref{eq: choquard eq2}) on $G$.

(4) In the continuous setting \cite{MV13}, the existence range of
ground state solutions $\frac{N+\alpha}{N}<p<\frac{N+\alpha}{N-2}$
is sharp, in the sense that if $p\leq\frac{N+\alpha}{N}$ or $p\geq\frac{N+\alpha}{N-2}$
problem (\ref{eq: classical choquard equation}) does not have any
nontrivial variational solution. It follows from the Pohozaev identity
and the scaling trick on the Euclidean space $\mathbb{R}^{N}$, see
\cite[Theorem 2,  Proposition 3.1.]{MV13}, which are unknown on Cayley
graphs $G$. This leads to an open problem for the nonexistence of
nontrivial solutions to the equation (\ref{eq: choquard eq2}) on
$G$, see Conjecture \ref{conj: conj3}.

(5) Since the nonvanishing of the limit of the translation sequence
based on the fact that the parameter $p>\frac{N+\alpha}{N-2}$ is
supercritical, see Lemma \ref{lem:lower bound}, the existence of
nontrivial solutions in the critical case $p=\frac{N+\alpha}{N-2}$
is still an open problem, see Conjecture \ref{conj: conj3}.
\end{rem}

\begin{conjecture}
\label{conj: conj3}According to the results in continuous cases \cite[Corollary I.2]{L3}\cite{MV13},
we conjecture that (\ref{eq: choquard eq2}), (\ref{eq:R_a system}),
(\ref{eq:R_a 2}) and (\ref{eq:R_a system 2}) have positive solutions
when $p=\frac{N+\alpha}{N-2}$, and the non-negative solutions are
trivial when $p<\frac{N+\alpha}{N-2}$.
\end{conjecture}

According to \cite{HW,C3}, we define the $p$-Laplace of $u$ for
$p>1$ as
\[
\Delta_{p}u(x)\coloneqq\underset{y\sim x}{\sum}\lvert\nabla_{xy}u\rvert^{p-2}\nabla_{xy}u,
\]
and for $p=1$ as
\[
\Delta_{1}u(x)\coloneqq\left\{ \sum\limits _{y\sim x}f_{xy}:f_{xy}=-f_{yx},f_{xy}\in\textrm{Sgn}(\nabla_{xy}u)\right\} \text{, Sgn\ensuremath{\left(t\right)}=\ensuremath{\begin{cases}
\begin{array}{c}
\left\{ 1\right\} ,\\{}
[-1,1]\\
\left\{ -1\right\} ,
\end{array}, & \begin{array}{c}
t>0,\\
t=0,\\
t<0.
\end{array}\end{cases}}}
\]
And since the second-order Sobolev inequalities hold on $G$ \cite[Theorem 10]{HLM22},
by the same argument we can prove the following results.
\begin{thm}
\label{thm: main2}For $N\geq3$, $\alpha\in(0,N-2)$, $1\leq p<\frac{N-\alpha}{2}$,
there exists a positive ground state solution for $p>1$ and a non-negative
solution for $p=1$ of the equation 
\begin{equation}
\Delta_{p}u+(R_{\alpha}\ast u{}^{p})u^{p-1}=0.\label{eq: p-laplace}
\end{equation}

For $N\geq5$, $\alpha\in(0,N)$, $p>\frac{N+\alpha}{N-4}$, there
exists a positive ground state solution of the equation
\begin{equation}
\Delta^{2}u-(R_{\alpha}\ast u{}^{p})u^{p-1}=0.\label{eq: bi-harmonic}
\end{equation}

For $N\geq5$, $\alpha\in(0,N-4)$, $1<p<\frac{N-\alpha}{4}$, there
exists a positive ground state solution of the equation
\begin{equation}
\Delta\left(\mid\Delta u\mid^{p-2}\Delta u\right)-(R_{\alpha}\ast u{}^{p})u^{p-1}=0.\label{eq: p-bi-harmonic}
\end{equation}
\end{thm}

\begin{rem}
\label{rem:generalize 2}The systems of (\ref{eq: p-laplace}), (\ref{eq: bi-harmonic})
and (\ref{eq: p-bi-harmonic}) also have positive solutions.
\end{rem}

\begin{conjecture}
Also we conjecture that (\ref{eq: p-laplace}), (\ref{eq: bi-harmonic})
and (\ref{eq: p-bi-harmonic}) have positive solutions when $p=\frac{N-\alpha}{2}$,
$p=\frac{N+\alpha}{N-4}$ and $p=\frac{N-\alpha}{4}$ respectively,
and the non-negative solutions are trivial when $p>\frac{N-\alpha}{2}$,
$p<\frac{N+\alpha}{N-4}$ and $p>\frac{N-\alpha}{4}$ respectively.
\end{conjecture}

The paper is organized as follows. In Section 2, we recall some basic
facts and prove some useful lemmas. In Section 3, we prove the discrete
Concentration-Compactness principle and a key lemma to exclude the
vanishing case, see Lemma \ref{lem:Concentration-Compact} and Lemma
\ref{lem:lower bound}. In Section 4, we give two proofs for Theorem
\ref{thm:main1} and prove Corollary \ref{cor: Coro}. In Section
5, we prove Theorem \ref{thm: main2} which generalize the existence
results to the Choquard type equations with $p$-Laplace, biharmonic
and $p$-biharmonic operators.

\section{Preliminaries}

Let $G$ be a countable group. It is called finitely generated if
it has a finite generating set $S$. We always assume that the generating
set $S$ is symmetric, i.e. $S=S^{-1}$. The Cayley graph of $(G,S)$
is a graph structure $(V,E)$ with the set of vertices $V=G$ and
the set of edges $E$ where for any $x,y\text{\ensuremath{\in}}G,xy\in E$
(also denoted by $x\sim y$) if $x=ys$ for some $s\in S$. The Cayley
graph of $(G,S)$ is endowed with a natural metric, called the word
metric \cite{BBI01}: For any $x,y\in G$, the distance between them
is defined as the length of the shortest path connecting $x$ and
$y$ by assigning each edge of length one, 
\[
d^{S}(x,y)=\text{inf}{\left\{ k:x=x_{0}\sim\cdot\cdot\cdot\sim x_{k}=y\right\} .}
\]
One sees easily that for two generating sets $S$ and $S_{1}$ the
metrics $d^{S}$ and $d^{S_{1}}$ are bi-Lipschitz equivalent, i.e.
there exist two constants $C_{1}(S,S_{1}),C_{2}(S,S_{1})$ such that
for any $x,y\in G$ 
\[
C_{1}(S,S_{1})d^{S_{1}}(x,y)\leq d^{S}(x,y)\leq C_{2}(S,S_{1})d^{S_{1}}(x,y).
\]
Let $B_{p}^{S}(n):={\left\{ x\in G\mid d^{S}(p,x)\leq n\right\} }$
denote the closed ball of radius $n$ centered at $p\in G$. By the
group structure, it is obvious that $\mid B_{p}^{S}(n)\mid=\mid B_{q}^{S}(n)\mid$,
for any $p,q\in G$. The growth function of $(G,S)$ is defined as
$\beta_{S}(n):=\mid B_{e}^{S}(n)\mid$ where $e$ is the unit element
of $G$. A group $G$ is called of polynomial growth if there exists
a finite generating set $S$ such that $\beta_{S}(n)\leq Cn^{A}$
for some $C,A>0$ and any $n\geq1$. One checks that this definition
is independent of the choice of the generating set $S$. Thus, the
polynomial growth is indeed a property of the group $G$. In this
paper, we consider the Cayley graph $(G,S)$ of a group of polynomial
growth 
\[
C_{1}(S)n^{N}\leq\beta_{S}(n)\leq C_{2}(S)n^{N},
\]
for some $N\in\mathbb{N}$ and any $n\geq1$, where $N$ is called
the homogenous dimension of $G$.

We denote by $C(G)$ the space of functions on $G$. The support of
$u\in C(G)$ is defined as $\textrm{supp}(u)\coloneqq\{x\in G:u(x)\neq0\}$.
Let $C_{0}(G)$ be the set of all functions with finite support. For
any $u\in C(G)$, the $\ell^{p}$ norm of $u$ is defined as 
\[
\lVert u\rVert_{\ell^{p}(G)}\coloneqq\begin{cases}
\left(\underset{x\in G}{\sum}\lvert u(x)\rvert^{p}\right)^{1/p} & \text{\ensuremath{0<p<\infty,}}\\
\underset{x\in G}{\sup}\lvert u(x)\rvert & p=\infty,
\end{cases}
\]
and we shall write $\lVert u\rVert_{{\ell^{p}(G)}}$ as $\lVert u\rVert_{{p}}$
for convenience. The $\ell^{p}(G)$ space is defined as 
\[
\ell^{p}(G)\coloneqq\left\{ u\in C(G):\lVert u\rVert_{{\ell^{p}(G)}}<\infty\right\} .
\]
For any $u\in C(G)$, the difference operator is defined as for any
$x\sim y$

\[
\nabla_{xy}u=u(y)-u(x).
\]
Let 
\[
\lvert\nabla u(x)\rvert_{p}\coloneqq\left(\sum\limits _{y\sim x}\lvert\nabla_{xy}u\rvert^{p}\right)^{1/p}
\]
be the $p$-norm of the gradient of $u$ at $x$.

We define the Laplace operator as 
\[
\Delta u(x)\coloneqq\sum\limits _{y\sim x}(u(y)-u(x)).
\]
The $D^{k,p}$ ($k=1,2$) norms of $u$ are given by 
\[
\lVert u\rVert_{D^{1,p}(G)}\coloneqq\lVert\lvert\nabla u\rvert_{p}^{p}\rVert_{\ell^{1}(G)}^{1/p}=\left(\sum\limits _{x\in G}\sum\limits _{y\sim x}\lvert\nabla_{xy}u\rvert^{p}\right)^{1/p},
\]
\[
\lVert u\rVert_{D^{2,p}(G)}\coloneqq\lVert\Delta u\rVert_{\ell^{p}(G)}=\left(\sum\limits _{x\in G}\lvert\sum\limits _{y\sim x}\nabla_{xy}u\rvert^{p}\right)^{1/p}.
\]
We define $D_{0}^{k,p}(G)$ ($k=1,2$) as the completion of $C_{0}\left(G\right)$
in the $D^{k,p}$ norm and define 
\[
D^{k,p}(G)\coloneqq\left\{ u\in\ell^{\frac{Np}{N-kp}}(G):\lVert u\rVert_{D^{k,p}(G)}<\infty\right\} .
\]
In fact $D_{0}^{k,p}(G)=D^{k,p}(G)$, see \cite[Section 3]{HLM22}
for a proof.

Let $\Omega$ be a subset of $G$. We denoted by 
\[
\delta\Omega\coloneqq\left\{ x\in G\backslash\Omega:\exists y\in\Omega,s.t.\;x\sim y\right\} 
\]
the vertex boundary of $\Omega$, possibly an empty set. We set $\bar{\Omega}\coloneqq\Omega\cup\delta\Omega$.

Let $c_{0}(G)$ be the completion of $C_{0}\left(G\right)$ in $\ell^{\infty}$
norm. It is well-known that $\ell^{1}(G)=(c_{0}(G))^{\ast}$. We set
\[
\lVert\mu\rVert\coloneqq\sup\limits _{u\in c_{0}(G),\lVert u\rVert_{\infty}=1}\langle\mu,u\rangle,\qquad\forall\mu\in\ell^{1}(G).
\]
By definition, 
\[
\mu_{n}\xrightharpoonup{w^{\ast}}\mu\;\text{in}\;\ell^{1}(G)\;\text{if and only if}\;\langle\mu_{n},u\rangle\longrightarrow\langle\mu,u\rangle,\forall u\in c_{0}(G).
\]
In the proof, we will use the following facts (see \cite{C90}). 
\begin{fact}
(a) Every bounded sequence of $\ell^{1}(G)$ contains a $w^{\ast}$-convergent
subsequence. \\
 (b) If $\mu_{n}\xrightharpoonup{w^{\ast}}\mu$ in $\ell^{1}(G)$,
then ${\mu_{n}}$ is bounded and 
\[
\lVert\mu\rVert\leq\varliminf\limits _{n\to\infty}\lVert\mu_{n}\rVert.
\]
(c) If $\mu\in\ell^{1+}(G):=\left\{ \mu\in\ell^{1}(G):\mu\geq0\right\} $,
then 
\[
\lVert\mu\rVert=\langle\mu,1\rangle.
\]
\end{fact}

Next, we prove the discrete HLS inequalities on $G$.
\begin{thm}
\label{thm: hls on G}For $r,s>1,$ $0<\alpha<N$, $\frac{1}{r}+\frac{1}{s}+\frac{N-\alpha}{N}=2$,
we get the discrete HLS inequality

\begin{equation}
\sum_{\substack{x,y\in G\\
x\neq y
}
}R_{\alpha}(x,y)f(x)g(y)\leq C_{r,s,\alpha}\Vert f\Vert_{\ell^{r}}\Vert g\Vert_{\ell^{s}},\:\forall f\in\ell^{r}(G),g\in\ell^{s}(G),\label{eq:discrere hls-1}
\end{equation}
and an equivalent form is

\begin{equation}
\lVert R_{\alpha}\ast f\rVert_{\ell^{t}}\leq C_{r,\alpha}\lVert f\rVert_{\ell^{r}},\:\forall f\in\ell^{r}(G),\label{eq:hls2-1}
\end{equation}
where $1<r<\frac{N}{\alpha},0<\alpha<N,t=\frac{Nr}{N-\alpha r}$.
\end{thm}

\begin{proof}
By the method of subordination and Bochner\textquoteright s functional
calculus \cite{SSV12,BBKRSV80}, the fractional Laplace operator on
$G$ is defined as
\[
(-\Delta)^{\frac{\alpha}{2}}u\coloneqq\sideset{\frac{1}{\mid\varGamma(-\frac{\alpha}{2})\mid}}{_{0}^{\infty}}\int\left(e^{t\Delta}u-u\right)t^{-1-\frac{\alpha}{2}}\text{d}t,
\]
where $e^{t\Delta}$ is the semigroup of $\Delta$, see \cite{KL12}.
Since $G$ with the homogenous dimension $N$ satisfies the weak Poincaré
inequality, the heat kernel $k_{t}(x,y)$ of the Laplace operator
on $G$ has the Gaussian heat kernel bounds \cite[Theorem 6.19]{B17}.
Hence the Green's function $R_{\alpha}$ of the fractional Laplace
is
\[
R_{\alpha}(x,y)=\frac{1}{\Gamma(\frac{\alpha}{2})}\int_{0}^{\infty}k_{t}(x,y)t^{-1+\frac{\alpha}{2}}\text{d\ensuremath{t}},\;x,y\in G,
\]
which has the asymptotic relation $R_{\alpha}(x,y)\simeq\left(d^{S}(x,y)\right)^{\alpha-N}$.
Then by the Young\textquoteright s inequality for weak type spaces
\cite[Theorem 1.4.25.]{G14}, 
\[
\lVert R_{\alpha}\ast f\rVert_{\ell^{t}}\lesssim\lVert R_{\alpha}\rVert_{\ell^{\frac{N}{N-\alpha},\infty}}\lVert f\rVert_{\ell^{r}}\simeq\lVert f\rVert_{\ell^{r}},
\]
where $\lVert f\rVert_{\ell^{p,\infty}}\coloneqq\underset{t>0}{\text{sup }}t\lambda_{f}^{1/p}(t),\,\lambda_{f}(t)\coloneqq\sharp\left\{ x:\mid f(x)\mid>t\right\} $.
And (\ref{eq:discrere hls-1}) follows from the H$\ddot{o}$lder inequality.
\end{proof}
Then we introduce the classical Brézis-Lieb lemma \cite[Theorem 1]{BL}.
Consider a measure space $(\Omega,\Sigma,\mu)$, which consists of
a set $\Omega$ equipped with a $\sigma$-algebra $\Sigma$ and a
Borel measure $\mu:\Sigma\to\left[0,\infty\right]$. 
\begin{lem}
\label{lem:(Brezis-Lieb-lemma)}(Brézis-Lieb lemma) Let $(\Omega,\Sigma,\mu)$
be a measure space, $\{u_{n}\}\subset L^{p}(\Omega,\Sigma,\mu)$,
and $0<p<\infty.$ If \\
 (a) $\;\{u_{n}\}$ is uniformly bounded in $L^{p}$, \\
 (b) $\;u_{n}\to u,n\to\infty$ $\mu$-almost everywhere in $\Omega$,
then 
\begin{equation}
\lim\limits _{n\to\infty}(\lVert u_{n}\rVert_{L^{p}}^{p}-\lVert u_{n}-u\rVert_{L^{p}}^{p})=\lVert u\rVert_{L^{p}}^{p}.\label{eq:weak}
\end{equation}
\end{lem}

\begin{rem}
\label{rem:Br=0000E9zis-Lieb remark}(1) The preceding lemma is a
refinement of the Fatou's Lemma.

(2) Since $\;\{u_{n}\}$ is uniformly bounded in $L^{p}$, passing
to a subsequence if necessary, we have 
\[
\lim\limits _{n\to\infty}\lVert u_{n}\rVert_{p}^{p}=\lim\limits _{n\to\infty}\lVert u_{n}-u\rVert_{p}^{p}+\lVert u\rVert_{p}^{p}.
\]

(3) If $\Omega$ is countable and $\mu$ is a positive measure defined
on $\Omega$, then we get a discrete version of the classical Brézis-Lieb
lemma.

(4) If $p\geq1$, for every $q\in[1,p]$, $\{u_{n}\}\subset L^{p}$
still satisfies conditions (a) and (b), then there exists an easy
variant of the classical Brézis-Lieb lemma
\[
\lim\limits _{n\to\infty}\intop\mid\mid u_{n}\mid^{q}-\mid u_{n}-u\mid^{q}-\mid u\mid^{q}\mid^{\frac{p}{q}}\text{d\ensuremath{\mu}}=0.
\]

(5) For $p>1$, (a) and (b) yield $u_{n}\xrightharpoonup{w}u$ weakly
in $L^{p}(\Omega)$, that is, the pointwise convergence of a bounded
sequence yields the weak convergence, see \cite[Proposition 4.7.12]{B07}.
\end{rem}

Now we are ready to prove three discrete versions of the Brézis-Lieb
lemma.
\begin{lem}
\label{lem: Brezis=002013Lieb lemma-1}Let $\Omega\subset G$, ${\left\{ u_{n}\right\} }\subset D^{1,p}(G)$,
and $1\leq p<\infty$. If \\
 (a)$'$ $\left\{ {u_{n}}\right\} $ is uniformly bounded in $D^{1,p}(G)$,
\\
 (b)$'$ $u_{n}\to u,n\to\infty$ pointwise on $G$, then 
\begin{equation}
\lim\limits _{n\to\infty}\left(\sum\limits _{x\in\Omega}\lvert\nabla u_{n}(x)\rvert_{p}^{p}-\sum\limits _{x\in\Omega}\lvert\nabla(u_{n}-u)(x)\rvert_{p}^{p}\right)=\sum\limits _{x\in\Omega}\lvert\nabla u(x)\rvert_{p}^{p}.\label{eq:gradient}
\end{equation}
\end{lem}

\begin{proof}
We define two directed edge sets as follows 
\[
E_{1}\coloneqq\left\{ e=\left(e_{-},e_{+}\right):e_{\pm}\in\Omega,e_{-}\sim e_{+}\right\} ,
\]
\[
E_{2}\coloneqq\left\{ e=\left(e_{-},e_{+}\right):\ensuremath{(e_{-},e_{+})\in}\Omega\times\ensuremath{\delta\Omega},e_{-}\sim e_{+}\right\} ,
\]
where $E_{1}$ is the set of internal edges of $\Omega$, $E_{2}$
is the set of edges that cross the boundary of $\Omega$, and $e_{-}$
and $e_{+}$ are the initial and terminal endpoints of $e$.

Set $\widetilde{E}\coloneqq E_{1}\cup E_{2}$. We define $\phantom{}\overline{u},\mu:\widetilde{E}\to\mathbb{R}$,
$\phantom{}\overline{u}(e)=u(e_{+})-u(e_{-})$, $\mu(e)=1$.

Then we get 
\[
\sum\limits _{\Omega}\lvert\nabla u_{n}(x)\rvert_{p}^{p}=\sum\limits _{\,\,E_{1}}\lvert\overline{u}_{n}(e)\rvert^{p}+\sum\limits _{\,\,E_{2}}\lvert\overline{u}_{n}(e)\rvert^{p}=\lVert\overline{u}_{n}\rVert_{\ell^{p}(\widetilde{E},\Sigma,\mu)}^{p}<\infty,
\]
\[
\overline{u}_{n}\to\overline{u}\;\textrm{pointwise\;in}\;\widetilde{E}.
\]
For the measure space $(\widetilde{E},\Sigma,\mu)$, by the Brézis-Lieb
lemma we have 
\[
\lim\limits _{n\to\infty}(\lVert\overline{u}_{n}\rVert_{\ell^{p}(\widetilde{E},\Sigma,\mu)}^{p}-\lVert\overline{u_{n}-u}\rVert_{\ell^{p}(\widetilde{E},\Sigma,\mu)}^{p})=\lVert\overline{u}\rVert_{\ell^{p}(\widetilde{E},\Sigma,\mu)}^{p},
\]
which is equivalent to the equation (\ref{eq:gradient}). 
\end{proof}
\begin{lem}
\label{lem: Br=0000E9zis-Lieb lemma-3}Let $\Omega\subset G$, ${\left\{ u_{n}\right\} }\subset D^{2,p}(G)$,
and $1<p<\infty$. If \\
 (a)$''$ $\left\{ {u_{n}}\right\} $ is uniformly bounded in $D^{2,p}(G)$,\\
 (b)$''$ $u_{n}\to u,n\to\infty$ pointwise on $G$, then 
\begin{equation}
\lim\limits _{n\to\infty}\left(\sum\limits _{x\in\Omega}\lvert\Delta u_{n}(x)\rvert^{p}-\sum\limits _{x\in\Omega}\lvert\Delta(u_{n}-u)(x)\rvert^{p}\right)=\sum\limits _{x\in\Omega}\lvert\Delta u(x)\rvert^{p}.\label{eq:gradient-2}
\end{equation}
\end{lem}

\begin{proof}
Since $\left\{ {\Delta u_{n}}\right\} $ is uniformly bounded in $\ell^{p}(G)$
and $\Delta u_{n}\rightarrow\Delta u,n\rightarrow\infty$ pointwise
on $G$, by the Brézis-Lieb lemma we get 
\[
\lim\limits _{n\to\infty}(\lVert\Delta u_{n}\rVert_{\ell^{p}(\Omega)}^{p}-\lVert\Delta\left(u_{n}-u\right)\rVert_{\ell^{p}(\Omega)}^{p})=\lVert\Delta u\rVert_{\ell^{p}(\Omega)}^{p},
\]
which is equivalent to the equation (\ref{eq:gradient-2}). 
\end{proof}
\begin{lem}
\label{lem: Brezis=002013Lieb lemma-2}Let $\Omega\subset G$, for
$N\geq3$, $\alpha\in(0,N)$, $p\geq1$, and ${\left\{ u_{n}\right\} }\subset\ell^{\frac{2Np}{N+\alpha}}(G)$.
If \\
(a)$'''$ $\left\{ {u_{n}}\right\} $ is uniformly bounded in $\ell^{\frac{2Np}{N+\alpha}}(G)$,
\\
(b)$'''$ $u_{n}\to u,n\to\infty$ pointwise on $G$, then 
\begin{equation}
\lim\limits _{n\to\infty}\left(\sum_{\substack{\Omega\\
\\
}
}\left(R_{\alpha}\ast\lvert u_{n}\rvert^{p}\right)\lvert u_{n}\rvert^{p}-\sum_{\substack{\Omega\\
\\
}
}\left(R_{\alpha}\ast\lvert u_{n}-u\rvert^{p}\right)\lvert u_{n}-u\rvert^{p}\right)=\sum_{\substack{\Omega\\
\\
}
}\left(R_{\alpha}\ast\lvert u\rvert^{p}\right)\lvert u\rvert^{p}.\label{eq:I_a trem-Brezis-Lieb}
\end{equation}
\end{lem}

\begin{proof}
First we have 
\begin{align*}
 & \sum_{\substack{\Omega\\
\\
}
}\left(R_{\alpha}\ast\lvert u_{n}\rvert^{p}\right)\lvert u_{n}\rvert^{p}-\sum_{\substack{\Omega\\
\\
}
}\left(R_{\alpha}\ast\lvert u_{n}-u\rvert^{p}\right)\lvert u_{n}-u\rvert^{p}\\
= & \sum_{\substack{\Omega\\
\\
}
}\left(R_{\alpha}\ast\left(\lvert u_{n}\rvert^{p}-\lvert u_{n}-u\rvert^{p}\right)\right)\left(\lvert u_{n}\rvert^{p}-\lvert u_{n}-u\rvert^{p}\right)\\
 & +\sum_{\substack{\Omega\\
\\
}
}\left(R_{\alpha}\ast\lvert u_{n}\rvert^{p}\right)\lvert u_{n}-u\rvert^{p}+\sum_{\substack{\Omega\\
\\
}
}\left(R_{\alpha}\ast\lvert u_{n}-u\rvert^{p}\right)\lvert u_{n}\rvert^{p}\\
 & -2\sum_{\substack{\Omega\\
\\
}
}\left(R_{\alpha}\ast\lvert u_{n}-u\rvert^{p}\right)\lvert u_{n}-u\rvert^{p}\\
\eqqcolon & \text{\mbox{I}}+\text{\mbox{II}},
\end{align*}
where 
\[
\text{\mbox{I}}=\sum_{\substack{x\in\Omega\\
\\
}
}\left(R_{\alpha}\ast f_{n}(x)\right)f_{n}(x),\;f_{n}(x)\coloneqq\lvert u_{n}(x)\rvert^{p}-\lvert u_{n}(x)-u(x)\rvert^{p}
\]
and
\begin{align*}
\text{\mbox{II}} & =\sum_{\substack{\Omega\\
\\
}
}\left(R_{\alpha}\ast\lvert u_{n}\rvert^{p}\right)\lvert u_{n}-u\rvert^{p}+\sum_{\substack{\Omega\\
\\
}
}\left(R_{\alpha}\ast\lvert u_{n}-u\rvert^{p}\right)\lvert u_{n}\rvert^{p}\\
 & -2\sum_{\substack{\Omega\\
\\
}
}\left(R_{\alpha}\ast\lvert u_{n}-u\rvert^{p}\right)\lvert u_{n}-u\rvert^{p}.
\end{align*}
For \mbox{I}, by the HLS inequality (\ref{eq: hls1}) we get

\begin{align*}
 & \lvert\text{\mbox{I}}-\sum_{\substack{\Omega\\
\\
}
}\left(R_{\alpha}\ast\lvert u\rvert^{p}\right)\lvert u\rvert^{p}\lvert\\
 & =\lvert\sum_{\substack{\Omega\\
\\
}
}\left(R_{\alpha}\ast(f_{n}-\lvert u\rvert^{p})\right)f_{n}+\sum_{\substack{\Omega\\
\\
}
}\left(R_{\alpha}\ast\lvert u\rvert^{p}\right)(f_{n}-\lvert u\rvert^{p})\lvert\\
 & \leq\sum\limits _{\substack{G\\
\\
}
}\left(R_{\alpha}\ast\lvert f_{n}-\lvert u\rvert^{p}\lvert\right)\lvert f_{n}\lvert+\sum_{\substack{G\\
\\
}
}\left(R_{\alpha}\ast\lvert u\rvert^{p}\right)\lvert f_{n}-\lvert u\rvert^{p}\lvert\\
 & \lesssim\lVert f_{n}-\lvert u\rvert^{p}\rVert_{\frac{2N}{N+\alpha}}\lVert f_{n}\rVert_{\frac{2N}{N+\alpha}}+\lVert f_{n}-\lvert u\rvert^{p}\rVert_{\frac{2N}{N+\alpha}}\lVert u\rVert_{\frac{2Np}{N+\alpha}}^{p}.
\end{align*}
Since $\left\{ {u_{n}}\right\} $ is uniformly bounded in $\ell^{\frac{2Np}{N+\alpha}}$,
by Remark \ref{rem:Br=0000E9zis-Lieb remark} (4) we get $f_{n}\to\mid u\mid^{p}$
strongly in $\ell^{\frac{2N}{N+\alpha}}$ as $n\to\infty$. Hence
by letting $n\to\infty$,
\[
\text{\mbox{I}}\to\sum_{\substack{\Omega\\
\\
}
}\left(R_{\alpha}\ast\lvert u\rvert^{p}\right)\lvert u\rvert^{p}.
\]
For \mbox{II},
\begin{align*}
\lvert\text{\mbox{II}} & \lvert=\lvert\sum_{\substack{\Omega\\
\\
}
}\left(R_{\alpha}\ast f_{n}\right)\lvert u_{n}-u\rvert^{p}+\sum_{\substack{\Omega\\
\\
}
}\left(R_{\alpha}\ast\lvert u_{n}-u\rvert^{p}\right)f_{n}\lvert\\
 & \leq\sum\limits _{\substack{G\\
\\
}
}\left(R_{\alpha}\ast\lvert f_{n}\lvert\right)\lvert u_{n}-u\rvert^{p}+\sum\limits _{\substack{G\\
\\
}
}\left(R_{\alpha}\ast\lvert u_{n}-u\rvert^{p}\right)\lvert f_{n}\lvert\\
 & \leq2\sum\limits _{G}(R_{\alpha}\ast\mid f_{n}\mid)\lvert u_{n}-u\rvert^{p}.
\end{align*}
By the HLS inequality (\ref{eq:hls2}), Remark \ref{rem:Br=0000E9zis-Lieb remark}
(4) and (5),
\[
R_{\alpha}\ast\mid f_{n}\mid\to R_{\alpha}\ast\mid u\mid^{p}\:\text{in }\ensuremath{\ell^{\frac{2N}{N-\alpha}}},
\]
\[
\lvert u-u_{n}\rvert^{p}\xrightharpoonup{w}0\;\text{weakly in \ensuremath{\ell^{\frac{2N}{N+\alpha}}}},
\]
then we get $\text{\mbox{II}}\to0$ as $n\to\infty$. Hence we reach
the conclusion.
\end{proof}
\begin{rem}
Note that a counterpart of the Brézis-Lieb lemma holds for the nonlocal
term of the Riesz potential $I_{\alpha}$ on $\mathbb{R}^{N}$ \cite[Lemma 3.2]{YW13}\cite[§5.1]{A06}:
if the sequence $\left\{ u_{n}\right\} $ is uniformly bounded in
$L^{\frac{2Np}{N+\alpha}}(\mathbb{R}^{N})$ and $u_{n}\to u,n\to\infty$
almost everywhere on $\mathbb{R}^{N}$, then 
\[
\lim\limits _{n\to\infty}\left(\underset{\mathbb{R}^{N}}{\int}(I_{\alpha}\ast\mid u_{n}\mid^{p})\mid u_{n}\mid^{p}-\underset{\mathbb{R}^{N}}{\int}(I_{\alpha}\ast\mid u_{n}-u\mid^{p})\mid u_{n}-u\mid^{p}\right)=\underset{\mathbb{R}^{N}}{\int}(I_{\alpha}\ast\mid u\mid^{p})\mid u\mid^{p}.
\]
Here we prove a discrete version on $\Omega\subset G$, which is necessary
in the later proofs of the discrete Concentration-Compactness lemma.
\end{rem}

\section{Concentration-Compactness Principle}

In this section, we will establish the discrete Concentration-Compactness
principle and prove a key lemma to rule out the vanishing case of
the limit function.
\begin{lem}
\label{lem:Concentration-Compact}(Discrete Concentration-Compactness
lemma) For $N\geq3$, $\alpha\in(0,N)$, $p\geq\frac{N+\alpha}{N-2}$,
if $\{u_{n}\}$ is uniformly bounded in $D^{1,2}(G)$. Then passing
to a subsequence if necessary, still denoted as $\{u_{n}\}$, we have
\begin{equation}
u_{n}\to u\quad\textrm{pointwise\;on}\;G,\label{eq:pointwise}
\end{equation}

\begin{equation}
\lvert\nabla u_{n}\rvert_{2}^{2}\xrightharpoonup{w^{\ast}}\lvert\nabla u\rvert_{2}^{2}\quad\textrm{in}\;\ell^{1}(G).\label{eq:w-star}
\end{equation}
And the following limits 
\[
\lim\limits _{r\to\infty}\lim\limits _{n\to\infty}\sum\limits _{d^{S}(x,e)>r}\lvert\nabla u_{n}\left(x\right)\rvert_{2}^{2}\eqqcolon\mu_{\infty},\:\lim\limits _{r\to\infty}\lim\limits _{n\to\infty}\sum\limits _{\substack{d^{S}(x,e)>r\\
\\
}
}(R_{\alpha}\ast\mid u_{n}\mid^{p})\lvert u_{n}(x)\rvert^{p}\eqqcolon\nu_{\infty}
\]
exist. For the above $\left\{ u_{n}\right\} $, we have 
\begin{equation}
\lvert\nabla(u_{n}-u)\rvert_{2}^{2}\xrightharpoonup{w^{\ast}}0\quad\text{in}\;\ell^{1}(G),\label{eq:nu=00003D00003D00003D00003D00003D00003D0}
\end{equation}
\begin{equation}
(R_{\alpha}\ast\mid u_{n}-u\mid^{p})\mid u_{n}-u\mid^{p}\xrightharpoonup{w^{\ast}}0\quad\text{in}\;\ell^{1}(G),\label{eq:mu=00003D00003D00003D00003D00003D00003D0}
\end{equation}
\begin{equation}
\nu_{\infty}\leq K\mu_{\infty}^{p},\label{eq:infinite}
\end{equation}
\begin{equation}
\lim\limits _{n\to\infty}\lVert u_{n}\rVert_{D^{1,2}}^{2}=\lVert u\rVert_{D^{1,2}}^{2}+\mu_{\infty},\label{eq:composition1}
\end{equation}
\begin{equation}
\lim\limits _{n\to\infty}\sum(R_{\alpha}\ast\mid u_{n}\mid^{p})\mid u_{n}\mid^{p}=\sum(R_{\alpha}\ast\mid u\mid^{p})\mid u\mid^{p}+\nu_{\infty}.\label{eq:composition2}
\end{equation}
\end{lem}

\begin{proof}
Since $\{u_{n}\}$ is uniformly bounded in $\ell^{\frac{2Np}{N+\alpha}}(G)$,
and hence in $\ell^{\infty}(G)$. By the diagonal principle, passing
to a subsequence we get (\ref{eq:pointwise}). Since $\left\{ \lvert\nabla u_{n}\rvert_{2}^{2}\right\} $
is uniformly bounded in $\ell^{1}(G)$, we get (\ref{eq:w-star})
by the Banach-Alaoglu theorem and (\ref{eq:pointwise}). For every
$r\geq1$, passing to a subsequence if necessary, 
\[
\lim\limits _{n\to\infty}\sum\limits _{d^{S}(x,e)>r}\lvert\nabla u_{n}\left(x\right)\rvert_{2}^{2},\qquad\lim\limits _{n\to\infty}\sum\limits _{\substack{d^{S}(x,e)>r\\
\\
}
}(R_{\alpha}\ast\mid u_{n}\mid^{p})\lvert u_{n}(x)\rvert^{p}
\]
exist, where $d^{S}$ is the word metric as defined in Section 2.
Then we can define $\mu_{\infty}$, $\nu_{\infty}$ by the monotonicity
in $R$.

Let $v_{n}:=u_{n}-u$, then $v_{n}\to0$ pointwise on $G$ and $\left\{ \lvert\nabla v_{n}\rvert_{2}^{2}\right\} $
is uniformly bounded in $\ell^{1}(G)$. Then any subsequence of $\left\{ \lvert\nabla v_{n}\rvert_{2}^{2}\right\} $
contains a subsequence (still denoted as $\left\{ \lvert\nabla v_{n}\rvert_{2}^{2}\right\} $)
that $w^{\ast}$-converges to 0 in $\ell^{1}(G)$, which follows from
\[
\sum h\lvert\nabla v_{n}\rvert_{2}^{2}\to0,\text{ }\forall h\in C_{0}(G).
\]
Hence we get (\ref{eq:nu=00003D00003D00003D00003D00003D00003D0}).
Similarly, we get (\ref{eq:mu=00003D00003D00003D00003D00003D00003D0}).

For $r\geq1$, let $\Psi_{r}\in C(G)$ such that $\Psi_{r}(x)=1$
for $d^{S}(x,e)\geq r+1$, $\Psi_{r}(x)=0$ for $d^{S}(x,e)\leq r$.
By the HLS inequality (\ref{eq: hls1}), we have 
\begin{align}
(\sum(R_{\alpha}\ast\mid v_{n}\Psi_{r}\mid^{p})\mid v_{n}\Psi_{r}\mid^{p}) & \leq K\left(\sum\lvert\nabla(v_{n}\Psi_{r})\rvert^{2}\right)^{p}\label{eq:ineq1}\\
 & =K\left(\underset{x}{\sum}\underset{y\sim x}{\sum}\lvert\nabla_{xy}\Psi_{r}v_{n}(y)+\Psi_{r}(x)\nabla_{xy}v_{n}\rvert^{2}\right)^{p}.\nonumber 
\end{align}
For the left side of (\ref{eq:ineq1}), by the definition of $\Psi_{r}$,
the H$\ddot{o}$lder inequality and the HLS inequality, we have the
following key observation
\begin{align*}
 & \mid\sum(R_{\alpha}\ast\mid v_{n}\Psi_{r}\mid^{p})\mid v_{n}\Psi_{r}\mid^{p}-\sum(R_{\alpha}\ast\mid v_{n}\mid^{p})\mid\Psi_{r}\mid^{p}\mid v_{n}\Psi_{r}\mid^{p}\mid\\
\simeq & \mid\sum\limits _{\substack{x\neq y\\
\\
}
}\frac{\left(\mid\Psi_{r}(y)\mid^{p}-\mid\Psi_{r}(x)\mid^{p}\right)\lvert v_{n}(y)\rvert^{p}\lvert v_{n}(x)\rvert^{p}\lvert\Psi_{r}(x)\rvert^{p}}{d^{S}(x,y){}^{N-\alpha}}\lvert\\
= & \mid\sum\limits _{\substack{d^{S}(x,e)\geq r+1\\
\\
}
}\sum\limits _{\substack{d^{S}(y,e)\leq r\\
\\
}
}\frac{\left(0-1\right)\lvert v_{n}(y)\rvert^{p}\lvert v_{n}(x)\rvert^{p}}{d^{S}(x,y){}^{N-\alpha}}\lvert\\
\lesssim & \sum\limits _{\substack{d^{S}(y,e)\leq r\\
\\
}
}\sum\limits _{\substack{x\neq y\\
\\
}
}\frac{\lvert v_{n}(y)\rvert^{p}\lvert v_{n}(x)\rvert^{p}}{d^{S}(x,y){}^{N-\alpha}}\\
\lesssim & \lVert R_{\alpha}\ast\lvert v_{n}\rvert^{p}\rVert_{\ell^{\frac{2N}{N-\alpha}}(\{y:d^{S}(y,e)\leq r\})}\lVert v_{n}{}^{p}\rVert_{\ell^{\frac{2N}{N+\alpha}}(\{y:d^{S}(y,e)\leq r\})}\\
\lesssim & \ensuremath{\lVert v_{n}{}^{p}\rVert_{\ell^{\frac{2N}{N+\alpha}}(G)}\lVert v_{n}{}^{p}\rVert_{\ell^{\frac{2N}{N+\alpha}}(\{y:d^{S}(y,e)\leq r\})},}
\end{align*}
which goes to zero since $v_{n}\to0$ pointwise on $G$. Thus we get
\begin{equation}
\lim\limits _{n\to\infty}\left(\sum(R_{\alpha}\ast\mid v_{n}\Psi_{r}\mid^{p})\mid v_{n}\Psi_{r}\mid^{p}-\sum(R_{\alpha}\ast\mid v_{n}\mid^{p})\mid\Psi_{r}\mid^{p}\mid v_{n}\Psi_{r}\mid^{p}\right)=0.\label{eq:claim of cutoff fun}
\end{equation}
For the right side of (\ref{eq:ineq1}), note that for any $\varepsilon>0$,
there exists $C_{\varepsilon}>0$ such that 
\[
\lvert\nabla_{xy}\Psi_{r}v_{n}(y)+\Psi_{r}(x)\nabla_{xy}v_{n})\rvert^{2}\leq C_{\varepsilon}\lvert\nabla_{xy}\Psi_{r}\lvert^{2}\lvert v_{n}(y)\lvert^{2}+(1+\varepsilon)\lvert\nabla_{xy}v_{n}\lvert^{2}\lvert\Psi_{r}(x)\lvert^{2}.
\]
Since $v_{n}\to0$ pointwise on $G$, by (\ref{eq:ineq1}), (\ref{eq:claim of cutoff fun})
and letting $\varepsilon\to0^{+}$ we obtain 
\[
\varlimsup\limits _{n\to\infty}\sum(R_{\alpha}\ast\mid v_{n}\mid^{p})\mid\Psi_{r}\mid^{p}\mid v_{n}\Psi_{r}\mid^{p}\leq K\varlimsup\limits _{n\to\infty}\left(\sum\lvert\nabla v_{n}\rvert_{2}^{2}\Psi_{r}^{2}\right)^{p}.
\]
From the definition of $\Psi_{r}$, 
\begin{equation}
\varlimsup\limits _{n\to\infty}\sum\limits _{\substack{d^{S}(x,e)>r\\
\\
}
}(R_{\alpha}\ast\mid v_{n}\mid^{p})\mid v_{n}\mid^{p}\leq K\varlimsup\limits _{n\to\infty}\left(\sum\limits _{d^{S}(x,e)>r}\lvert\nabla v_{n}\rvert_{2}^{2}\right)^{p}.\label{eq: c-c3}
\end{equation}
By Lemma \ref{lem: Brezis=002013Lieb lemma-1}, Lemma \ref{lem: Brezis=002013Lieb lemma-2}
and Remark \ref{rem:Br=0000E9zis-Lieb remark} (2),
\[
\lim\limits _{n\to\infty}\sum\limits _{d^{S}(x,e)>r}\lvert\nabla u_{n}(x)\rvert_{2}^{2}-\lim\limits _{n\to\infty}\sum\limits _{d^{S}(x,e)>r}\lvert\nabla v_{n}(x)\rvert_{2}^{2}=\sum\limits _{d^{S}(x,e)>r}\lvert\nabla u(x)\rvert_{2}^{2},
\]
\begin{align*}
\lim\limits _{n\to\infty}\sum\limits _{\substack{d^{S}(x,e)>r\\
\\
}
}(R_{\alpha}\ast\mid u_{n}\mid^{p})\mid u_{n}\mid^{p}-\lim\limits _{n\to\infty}\sum\limits _{\substack{d^{S}(x,e)>r\\
\\
}
}(R_{\alpha}\ast\mid v_{n}\mid^{p})\mid v_{n}\mid^{p}\\
=\sum\limits _{\substack{d^{S}(x,e)>r\\
\\
}
}(R_{\alpha}\ast\mid u\mid^{p})\mid u\mid^{p}.
\end{align*}
Hence by the above equalities, we get
\begin{equation}
\lim\limits _{r\to\infty}\lim\limits _{n\to\infty}\sum\limits _{d^{S}(x,e)>r}\lvert\nabla v_{n}(x)\rvert_{2}^{2}=\mu_{\infty},\label{eq:c-c4}
\end{equation}
\begin{equation}
\lim\limits _{r\to\infty}\lim\limits _{n\to\infty}\sum\limits _{\substack{d^{S}(x,e)>r\\
\\
}
}(R_{\alpha}\ast\mid v_{n}\mid^{p})\mid v_{n}\mid^{p}=\nu_{\infty}.\label{eq:c-c5}
\end{equation}
Combining the equations (\ref{eq: c-c3}), (\ref{eq:c-c4}) and (\ref{eq:c-c5}),
we get 
\[
\nu_{\infty}\leq K\mu_{\infty}^{p}.
\]

Since $u_{n}\to u$ pointwise on $G$, for every $r\geq1$, we have
\begin{align}
\varlimsup\limits _{n\to\infty}\sum\lvert\nabla u_{n}\rvert_{2}^{2} & =\varlimsup\limits _{n\to\infty}\left(\sum\Psi_{r}\lvert\nabla u_{n}\rvert_{2}^{2}+\sum(1-\Psi_{r})\lvert\nabla u_{n}\rvert_{2}^{2}\right)\label{eq:composition1-1}\\
 & =\varlimsup\limits _{n\to\infty}\sum\Psi_{r}\lvert\nabla u_{n}\rvert_{2}^{2}+\sum(1-\Psi_{r})\lvert\nabla u\rvert_{2}^{2}.\nonumber 
\end{align}
Since $\left\{ (R_{\alpha}\ast\mid u_{n}\mid^{p})\mid u_{n}\mid^{p}\right\} $
is uniformly bounded in $\ell^{1}(G)$ and by (\ref{eq:pointwise}),
\[
(R_{\alpha}\ast\mid u_{n}\mid^{p})\mid u_{n}\mid^{p}\xrightharpoonup{w^{\ast}}(R_{\alpha}\ast\mid u\mid^{p})\mid u\mid^{p}\quad\text{in}\;\ell^{1}(G).
\]
Similarly,
\begin{align}
 & \varlimsup\limits _{n\to\infty}\sum\limits _{\substack{\\
\\
}
}(R_{\alpha}\ast\mid u_{n}\mid^{p})\mid u_{n}\mid^{p}\label{eq:composition2-1}\\
 & =\varlimsup\limits _{n\to\infty}\left(\sum\Psi_{r}(R_{\alpha}\ast\mid u_{n}\mid^{p})\mid u_{n}\mid^{p}+\sum(1-\Psi_{r})(R_{\alpha}\ast\mid u_{n}\mid^{p})\mid u_{n}\mid^{p}\right)\nonumber \\
 & =\varlimsup\limits _{n\to\infty}\sum\Psi_{r}(R_{\alpha}\ast\mid u_{n}\mid^{p})\mid u_{n}\mid^{p}+\sum(1-\Psi_{r})(R_{\alpha}\ast\mid u\mid^{p})\mid u\mid^{p}.\nonumber 
\end{align}
Letting $r\to\infty$ in (\ref{eq:composition1-1}) and (\ref{eq:composition2-1}),
we obtain 
\[
\lim\limits _{n\to\infty}\sum\lvert\nabla u_{n}\rvert_{2}^{2}=\mu_{\infty}+\sum\lvert\nabla u\rvert_{2}^{2}=\mu_{\infty}+\lVert u\rVert_{D^{1,2}}^{2},
\]
\[
\lim\limits _{n\to\infty}\sum\limits _{\substack{\\
\\
}
}(R_{\alpha}\ast\mid u_{n}\mid^{p})\mid u_{n}\mid^{p}=\nu_{\infty}+\sum\limits _{\substack{\\
\\
}
}(R_{\alpha}\ast\mid u\mid^{p})\mid u\mid^{p}.
\]
\end{proof}
\begin{rem}
(1) In the continuous setting, P. L. Lions \cite{L3}, Bianchi et
al.\cite{BCS} and Ben-Naoum et al.\cite{BTW} proved that the mass
of a weakly convergent sequence can be divided into three parts, i.e.
the mass of the limit, the mass concentrated at finite points and
the loss of mass of the sequence at infinity. The corresponding parts
still satisfy the Sobolev inequalities, also see \cite[Lemma 1.40]{W96}\cite{GdEYZ20}. 

(2) The part of the mass concentrated at finite points vanish on $G$,
i.e. (\ref{eq:nu=00003D00003D00003D00003D00003D00003D0}) and (\ref{eq:mu=00003D00003D00003D00003D00003D00003D0}),
which may be not true in the continuous setting. For example, consider
the sequence of probability measures $\left\{ \delta_{n}\right\} $
in $[0,1]$, where $\delta_{n}(x)\coloneqq n\chi_{[0,\frac{1}{n}]}\textrm{d}x$,
then $\delta_{n}\to0$ almost everywhere in $[0,1]$. However, $\delta_{n}\xrightharpoonup{w^{\ast}}\delta_{0}$
in $\left(C[0,1]\right)^{\ast}$ and the Dirac measure $\delta_{0}$
is non-zero. This is the advantage of the discrete setting. 

(3) The part of the loss of mass at infinity satisfies the inequality
(\ref{eq:infinite}) which may be not true in the continuous setting,
since the claim (\ref{eq:claim of cutoff fun}) only holds for smooth
compact support functions $\Psi_{r}\in C_{0}^{\infty}(\mathbb{R}^{N})$,
see \cite[(2.8)]{GdEYZ20}. The key point is that the convergence
pointwise in a bounded domain of $G$ yields $\ell^{1}$ convergence
which may be not true on $\mathbb{R}^{N}$, and $\delta_{n}(x)$ in
(2) is a counter example. 
\end{rem}

Next, we prove that the maximizing sequence after translation has
a uniform positive lower bound at the unit element $e\in G$. This
is crucial to rule out the vanishing case of the limit function in
supercritical cases. 
\begin{lem}
\label{lem:lower bound}For $N\geq3$, $\alpha\in(0,N)$, $p>\frac{N+\alpha}{N-2}$,
let $\left\{ u_{n}\right\} \subset D^{1,2}(G)$ be a maximizing sequence
satisfying (\ref{eq:max seq}). Then $\varliminf\limits _{n\to\infty}\lVert\nabla u_{n}\rVert_{\ell^{\infty}}=\varliminf\limits _{n\to\infty}\underset{x}{\text{sup}}\underset{y\sim x}{\sum}\mid\nabla_{xy}u_{n}\mid>0$.
\end{lem}

\begin{proof}
Since $p>\frac{N+\alpha}{N-2}$ yields $\frac{2Np}{N+\alpha}>2^{\ast}$,
let $r^{\ast}\coloneqq\frac{2Np}{N+\alpha}>2^{\ast}$, hence $r>2$
and $r=\frac{2Np}{N+\alpha+2p}\in(\frac{Np}{N+p},\frac{2Np}{N+2p})$.
Choosing $q$ such that $2<q<r<\infty$, by the interpolation inequality
we have 
\[
C_{\alpha,p,q}\left(\sum(R_{\alpha}\ast\mid u_{n}\mid^{p})\mid u_{n}\mid^{p}\right)^{\frac{q}{2p}}\leq\lVert\nabla u_{n}\rVert_{q}^{q}\leq\lVert\nabla u_{n}\rVert_{2}^{2}\lVert\nabla u_{n}\rVert_{\infty}^{q-2}\leq\lVert\nabla u_{n}\rVert_{\infty}^{q-2},
\]
where $C_{\alpha,p,q}$ is the constant in the inequality (\ref{eq: main inequality}).

By taking the limit, we obtain 
\[
C_{\alpha,p,q}K^{\frac{q}{2p}}\leq\varliminf\limits _{n\to\infty}\lVert\nabla u_{n}\rVert_{\infty}^{q-2}.
\]
This proves the lemma. 
\end{proof}
\begin{rem}
\label{rem:translation sequence}The maximum of $\lvert\nabla u_{n}(x)\rvert$
is attainable since $\lVert\nabla u_{n}\rVert_{2}=1$. Define $v_{n}\left(x\right):=u_{n}(x_{n}x)$,
where $\lvert\nabla u_{n}(x_{n})\rvert=\max\limits _{x}\lvert\nabla u_{n}(x)\rvert$.
Then the translation sequence $\left\{ v_{n}\right\} $ is uniformly
bounded in $D^{1,2}(G)$, $\lVert\nabla v_{n}\rVert_{2}=1$ and $\lvert\nabla v_{n}(e)\rvert=\lVert\nabla u_{n}\rVert_{\infty}$,
where $e$ is the unit element of $G$. By Lemma \ref{lem:lower bound},
passing to a subsequence if necessary, we have 
\[
v_{n}\to v\quad\;\textrm{pointwise\;on}\;G,
\]
\begin{equation}
\lvert\nabla v(e)\rvert=\varliminf\limits _{n\to\infty}\lVert\nabla u_{n}\rVert_{\ell^{\infty}}>0.\label{eq:v(0)>0}
\end{equation}
\end{rem}

\section{Proofs for Theorem \ref{thm:main1} and Corollary \ref{cor: Coro}}

In this section, we will give two proofs for Theorem \ref{thm:main1}
and prove Corollary \ref{cor: Coro}. 
\begin{proof}[Proof %
\mbox{%
I%
} of Theorem \ref{thm:main1}]
Let $\left\{ u_{n}\right\} \subset D^{1,2}(G)$ be a maximizing sequence
satisfying (\ref{eq:max seq}). And the translation sequence $\left\{ v_{n}\right\} $
is defined in Remark \ref{rem:translation sequence}.

By equalities (\ref{eq:composition2}) and (\ref{eq:composition1})
in Lemma \ref{lem:Concentration-Compact}, passing to a subsequence
if necessary, we get 
\[
1=\lim\limits _{n\to\infty}\lVert v_{n}\rVert_{D^{1,2}}^{2}=\lVert v\rVert_{D^{1,2}}^{2}+\mu_{\infty},
\]
\[
K=\lim\limits _{n\to\infty}\sum(R_{\alpha}\ast\mid v_{n}\mid^{p})\mid v_{n}\mid^{p}=\sum(R_{\alpha}\ast\mid v\mid^{p})\mid v\mid^{p}+\nu_{\infty}.
\]
By the Sobolev type inequality (\ref{eq: main inequality}), (\ref{eq:infinite})
and the inequality 
\begin{equation}
\left(a+b\right)^{p}\geq a^{p}+b^{p}\text{,}\;\forall a,b\geq0,\label{eq:trivial inequality}
\end{equation}
we get 
\begin{align*}
K & =\sum(R_{\alpha}\ast\mid v\mid^{p})\mid v\mid^{p}+\nu_{\infty}\\
 & \leq K(\lVert v\rVert_{D^{1,2}}^{2p}+\mu_{\infty}^{p})\\
 & \leq K(\lVert v\rVert_{D^{1,2}}^{2}+\mu_{\infty})^{p}=K.
\end{align*}
Since $\left(a+b\right)^{p}>a^{p}+b^{p}$ unless $a=0$ or $b=0$,
we deduce from (\ref{eq:v(0)>0}) that $\lVert v\rVert_{D^{1,2}}^{2}=1$
and $\mu_{\infty}=0.$ Hence $\nu_{\infty}=0$ which yields 
\[
\sum(R_{\alpha}\ast\mid v\mid^{p})\mid v\mid^{p}=K.
\]
That is, $v$ is a maximizer. 
\end{proof}
Next, we give another proof for Theorem \ref{thm:main1} using the
discrete Brézis-Lieb lemma. 
\begin{proof}[Proof %
\mbox{%
II%
} of Theorem \ref{thm:main1}]
Using Lemma \ref{lem:lower bound}, by the translation and taking
a subsequence if necessary as before, we can get a maximizing sequence
$\left\{ u_{n}\right\} $ satisfying (\ref{eq:max seq}), $u_{n}\to u$
pointwise on $G$, and $\lvert\nabla u(e)\rvert>0$.

By Lemma \ref{lem: Brezis=002013Lieb lemma-1}, Lemma \ref{lem: Brezis=002013Lieb lemma-2},
the inequalities (\ref{eq:trivial inequality}) and (\ref{eq: main inequality}),
then passing to a subsequence if necessary, we have 
\begin{align}
K & =\underset{n\rightarrow\infty}{\lim}\sum(R_{\alpha}\ast\mid u_{n}\mid^{p})\mid u_{n}\mid^{p}\nonumber \\
 & =\underset{n\rightarrow\infty}{\lim}\frac{\sum(R_{\alpha}\ast\mid u_{n}\mid^{p})\mid u_{n}\mid^{p}}{\lVert u_{n}\rVert_{D^{1,2}}^{2p}}\nonumber \\
 & =\underset{n\rightarrow\infty}{\overline{\lim}}\frac{\sum(R_{\alpha}\ast\mid u_{n}-u\mid^{p})\mid u_{n}-u\mid^{p}+\sum(R_{\alpha}\ast\mid u\mid^{p})\mid u\mid^{p}}{\left(\lVert u_{n}-u\rVert_{D^{1,2}}^{2}+\lVert u\rVert_{D^{1,2}}^{2}\right)^{p}}\label{eq: proof2}\\
 & \leq\underset{n\rightarrow\infty}{\overline{\lim}}\frac{\sum(R_{\alpha}\ast\mid u_{n}-u\mid^{p})\mid u_{n}-u\mid^{p}+\sum(R_{\alpha}\ast\mid u\mid^{p})\mid u\mid^{p}}{\lVert u_{n}-u\rVert_{D^{1,2}}^{2p}+\lVert u\rVert_{D^{1,2}}^{2p}}\nonumber \\
 & \leq\underset{n\rightarrow\infty}{\overline{\lim}}\frac{K\lVert u_{n}-u\rVert_{D^{1,2}}^{2p}+\sum(R_{\alpha}\ast\mid u\mid^{p})\mid u\mid^{p}}{\lVert u_{n}-u\rVert_{D^{1,2}}^{2p}+\lVert u\rVert_{D^{1,2}}^{2p}}.\nonumber 
\end{align}
Since $\lVert u\rVert_{D^{1,2}}\neq0$, we have that 
\[
\sum(R_{\alpha}\ast\mid u\mid^{p})\mid u\mid^{p}\geq K\lVert u\rVert_{D^{1,2}}^{2p},
\]
which yields 
\[
\sum(R_{\alpha}\ast\mid u\mid^{p})\mid u\mid^{p}=K\lVert u\rVert_{D^{1,2}}^{2p}.
\]
By (\ref{eq: proof2}), passing to a subsequence, we get
\[
\underset{n\rightarrow\infty}{\lim}\sum(R_{\alpha}\ast\mid u_{n}-u\mid^{p})\mid u_{n}-u\mid^{p}=K\underset{n\rightarrow\infty}{\lim}\lVert u_{n}-u\rVert_{D^{1,2}}^{2p}.
\]
Since $0<\lVert u\rVert_{D^{1,2}}\leq\underset{n\rightarrow\infty}{\lim}\lVert u_{n}\rVert_{D^{1,2}}=1$,
it suffices to show that $\lVert u\rVert_{D^{1,2}}=1$. Suppose that
it is not true, i.e. $0<\lVert u\rVert_{D^{1,2}}=D<1$, then by Lemma
\ref{lem: Brezis=002013Lieb lemma-1}, 
\[
\underset{n\rightarrow\infty}{\lim}\lVert u_{n}-u\rVert_{D^{1,2}}^{2}=\underset{n\rightarrow\infty}{\lim}\lVert u_{n}\rVert_{D^{1,2}}^{2}-\lVert u\rVert_{D^{1,2}}^{2}=1-D^{2}>0.
\]
However, $\left(a+b\right)^{p}>a^{p}+b^{p}$ if $a,$ $b>0$. This
yields a contradiction by (\ref{eq: proof2}).

Thus, $\lVert u\rVert_{D^{1,2}}=1$ and $u$ is a maximizer. 
\end{proof}
Finally we prove Corollary \ref{cor: Coro}. 
\begin{proof}[Proof of Corollary \ref{cor: Coro}]
 By Theorem \ref{thm:main1} there exists a maximizer $u$ for $K$.
Replacing $u$ by $\lvert u\rvert$, we know that $\lvert u\rvert$
is still a maximizer. Therefore, we get a non-negative maximizer $u$.
Define $I(u)\coloneqq\sum(R_{\alpha}\ast\mid u\mid^{p})\mid u\mid^{p}-\lambda\sum\mid\nabla u\mid^{2}$,
and it follows from the Lagrange multiplier that $u$ is a non-negative
solution of the equation (\ref{eq: choquard eq2}). The maximum principle
yields that $u$ is positive. 

Let $v\coloneqq R_{\alpha}\ast u{}^{p}$, which satisfies $\left(-\Delta\right)^{\frac{\alpha}{2}}v=u^{p}$.
Hence $(u,v)$ is a positive solution of the system (\ref{eq:R_a system}).
\end{proof}
By the equivalent form of  HLS inequality (\ref{eq:hls2}) we obtain
another inequality 
\[
\lVert R_{\alpha}\ast\mid u\mid^{p}\lVert_{^{\frac{2N}{N-\alpha}}}\lesssim\lVert u{}^{p}\lVert_{^{\frac{2N}{N+\alpha}}}\lesssim\lVert\nabla u\lVert_{2}^{p},\;\forall u\in D^{1,2}(G).
\]
Then define its variational problem, and by the same argument as before
we can prove Remark \ref{rem:equivalent form of HLS}.

\section{Proof for Theorem \ref{thm: main2}}

In this section, we generalize the existence results to the Choquard
type equations with $p$-Laplace, biharmonic and $p$-biharmonic operators.
By the Concentration-Compactness principle, we can prove the existence
of the extremal functions using the idea of proof I. Here we omit
the proof I and give the proof II below for brevity.
\begin{proof}[Proof of (\ref{eq: p-laplace})]
For $N\geq3$, $\alpha\in(0,N-2)$, $1\leq p<\frac{N-\alpha}{2}$,
then $\frac{2Np}{N+\alpha}>p^{\ast}=\frac{Np}{N-p}$, by the Sobolev
inequality (\ref{eq:discrete sobo}) and the HLS inequality (\ref{eq: hls1})
we get
\begin{align}
\sum_{G}\left(R_{\alpha}\ast\mid u\mid^{p}\right) & \mid u\mid^{p}\leq C_{p,\alpha}\left(\sum_{G}\mid u\mid^{\frac{2Np}{N+\alpha}}\right)^{\frac{N+\alpha}{N}}\label{eq: main inequality-2}\\
 & \leq C_{p,\alpha}\tilde{C}_{\frac{2Np}{N+\alpha},p}\left(\sum_{G}\mid\nabla u\mid^{p}\right)^{2},\;\forall u\in D^{1,p}(G).\nonumber 
\end{align}
The optimal constant in the inequality (\ref{eq: main inequality-2})
is given by 
\begin{equation}
K_{2}:=\underset{\lVert u\rVert_{D^{1,p}}=1}{\text{sup}}\sum_{G}\left(R_{\alpha}\ast\mid u\mid^{p}\right)\mid u\mid^{p}.\label{eq: sup-1}
\end{equation}
Consider a maximizing sequence $\{u_{n}\}\subset D^{1,p}(G)$ satisfying
\begin{equation}
\lVert u_{n}\rVert_{D^{1,p}}=1,\;\sum_{G}\left(R_{\alpha}\ast\mid u_{n}\mid^{p}\right)\mid u_{n}\mid^{p}\to K_{2},n\to\infty.\label{eq:max seq-1}
\end{equation}
Since $p<\frac{N-\alpha}{2}$ yields $\frac{2Np}{N+\alpha}\eqqcolon r^{\ast}>p^{\ast}$,
that is $r=\frac{Np}{N+\alpha+p}\in\left(\frac{Np}{2N+p},\frac{Np}{N+p}\right)$,
choose $q$ such that $p<q<r<\infty$. By the interpolation inequality
we have 
\[
C_{\alpha,p,q}\left(\sum(R_{\alpha}\ast\mid u_{n}\mid^{p})\mid u_{n}\mid^{p}\right)^{\frac{q}{2p}}\leq\lVert\nabla u_{n}\rVert_{q}^{q}\leq\lVert\nabla u_{n}\rVert_{p}^{p}\lVert\nabla u_{n}\rVert_{\infty}^{q-p}\leq\lVert\nabla u_{n}\rVert_{\infty}^{q-p},
\]
where $C_{\alpha,p,q}$ is the constant in the inequality (\ref{eq: main inequality-2}).

Taking the limit, we obtain 
\[
C_{\alpha,p,q}K_{2}^{\frac{q}{2p}}\leq\varliminf\limits _{n\to\infty}\lVert\nabla u_{n}\rVert_{\infty}^{q-p}.
\]
This proves $\varliminf\limits _{n\to\infty}\lVert\nabla u_{n}\rVert_{\ell^{\infty}}>0$.

By the translation and taking a subsequence if necessary as Remark
\ref{rem:translation sequence}, we can get a maximizing sequence
$\left\{ u_{n}\right\} $ satisfying (\ref{eq:max seq-1}), $u_{n}\to u$
pointwise on $G$, and $\lvert\nabla u(e)\rvert>0$.

By Lemma \ref{lem: Brezis=002013Lieb lemma-1}, Lemma \ref{lem: Brezis=002013Lieb lemma-2},
the inequalities (\ref{eq:trivial inequality}) and  (\ref{eq: main inequality-2}),
then passing to a subsequence if necessary, we have 
\begin{align}
K_{2} & =\underset{n\rightarrow\infty}{\lim}\sum(R_{\alpha}\ast\mid u_{n}\mid^{p})\mid u_{n}\mid^{p}\nonumber \\
 & =\underset{n\rightarrow\infty}{\lim}\frac{\sum(R_{\alpha}\ast\mid u_{n}\mid^{p})\mid u_{n}\mid^{p}}{\lVert u_{n}\rVert_{D^{1,p}}^{2p}}\nonumber \\
 & =\underset{n\rightarrow\infty}{\overline{\lim}}\frac{\sum(R_{\alpha}\ast\mid u_{n}-u\mid^{p})\mid u_{n}-u\mid^{p}+\sum(R_{\alpha}\ast\mid u\mid^{p})\mid u\mid^{p}}{\left(\lVert u_{n}-u\rVert_{D^{1,p}}^{p}+\lVert u\rVert_{D^{1,p}}^{p}\right)^{2}}\label{eq: proof2-1}\\
 & \leq\underset{n\rightarrow\infty}{\overline{\lim}}\frac{\sum(R_{\alpha}\ast\mid u_{n}-u\mid^{p})\mid u_{n}-u\mid^{p}+\sum(R_{\alpha}\ast\mid u\mid^{p})\mid u\mid^{p}}{\lVert u_{n}-u\rVert_{D^{1,p}}^{2p}+\lVert u\rVert_{D^{1,p}}^{2p}}\nonumber \\
 & \leq\underset{n\rightarrow\infty}{\overline{\lim}}\frac{K_{2}\lVert u_{n}-u\rVert_{D^{1,p}}^{2p}+\sum(R_{\alpha}\ast\mid u\mid^{p})\mid u\mid^{p}}{\lVert u_{n}-u\rVert_{D^{1,p}}^{2p}+\lVert u\rVert_{D^{1,p}}^{2p}}.\nonumber 
\end{align}
Since $\lVert u\rVert_{D^{1,p}}\neq0$, we have that 
\[
\sum(R_{\alpha}\ast\mid u\mid^{p})\mid u\mid^{p}\geq K_{2}\lVert u\rVert_{D^{1,p}}^{2p},
\]
which yields
\[
\sum(R_{\alpha}\ast\mid u\mid^{p})\mid u\mid^{p}=K_{2}\lVert u\rVert_{D^{1,p}}^{2p}.
\]
By (\ref{eq: proof2-1}), passing to a subsequence, we get
\[
\underset{n\rightarrow\infty}{\lim}\sum(R_{\alpha}\ast\mid u_{n}-u\mid^{p})\mid u_{n}-u\mid^{p}=K_{2}\underset{n\rightarrow\infty}{\lim}\lVert u_{n}-u\rVert_{D^{1,p}}^{2p}.
\]
Since $0<\lVert u\rVert_{D^{1,p}}\leq\underset{n\rightarrow\infty}{\lim}\lVert u_{n}\rVert_{D^{1,p}}=1$,
it suffices to show that $\lVert u\rVert_{D^{1,p}}=1$. Suppose that
it is not true, i.e. $0<\lVert u\rVert_{D^{1,p}}=D<1$, then by Lemma
\ref{lem: Brezis=002013Lieb lemma-1}, 
\[
\underset{n\rightarrow\infty}{\lim}\lVert u_{n}-u\rVert_{D^{1,p}}^{p}=\underset{n\rightarrow\infty}{\lim}\lVert u_{n}\rVert_{D^{1,p}}^{p}-\lVert u\rVert_{D^{1,p}}^{p}=1-D^{p}>0.
\]
However, $\left(a+b\right)^{2}>a^{2}+b^{2}$ if $a,$ $b>0$. This
yields a contradiction by (\ref{eq: proof2-1}).

Thus, $\lVert u\rVert_{D^{1,p}}=1$ and $u$ is a maximizer. Replacing
$u$ by $\lvert u\rvert$, we know that $\lvert u\rvert$ is still
a maximizer. Therefore, we get a non-negative maximizer $u$. It follows
from the Lagrange multiplier that $u$ is a non-negative solution
of (\ref{eq: p-laplace}). Moreover, for $p>1$, the maximum principle
yields that $u$ is positive. 
\end{proof}
In \cite[Theorem 10]{HLM22}, by the boundedness of Riesz transforms
we prove the discrete second-order Sobolev inequality (\ref{eq: second-order sobo}).
\begin{lem}
For $N\geq3,1<p<\dfrac{N}{2},p^{\ast\ast}\coloneqq\dfrac{Np}{N-2p}$,
\begin{equation}
\lVert u\rVert_{\ell^{p^{\ast\ast}}}\leq C_{p}\lVert u\rVert_{D^{2,p}},\;\forall u\in D^{2,p}(G).\label{eq: second-order sobo}
\end{equation}
\end{lem}

Then we are ready to prove (\ref{eq: bi-harmonic}).
\begin{proof}[Proof of (\ref{eq: bi-harmonic})]
 For $N\geq5$, $\alpha\in(0,N)$, $p>\frac{N+\alpha}{N-4}$, then
$\frac{2Np}{N+\alpha}>2^{\ast\ast}=\frac{2N}{N-4}$, by (\ref{eq: second-order sobo})
and the HLS inequality (\ref{eq: hls1}) we get 
\begin{align}
\sum_{G}\left(R_{\alpha}\ast\mid u\mid^{p}\right) & \mid u\mid^{p}\leq C_{p,\alpha}\left(\sum_{G}\mid u\mid^{\frac{2Np}{N+\alpha}}\right)^{\frac{N+\alpha}{N}}\label{eq: main inequality-3}\\
 & \leq C_{p,\alpha}\tilde{C}_{\frac{2Np}{N+\alpha}}\left(\sum_{G}\mid\Delta u\mid^{2}\right)^{p},\;\forall u\in D^{2,2}(G).\nonumber 
\end{align}

The optimal constant in the inequality (\ref{eq: main inequality-3})
is given by 
\begin{equation}
K_{3}:=\underset{\lVert u\rVert_{D^{2,2}}=1}{\text{sup}}\sum_{G}\left(R_{\alpha}\ast\mid u\mid^{p}\right)\mid u\mid^{p}.\label{eq: sup-3}
\end{equation}
And consider the maximizing sequence $\{u_{n}\}\subset D^{2,2}(G)$
satisfying

\begin{equation}
\lVert u_{n}\rVert_{D^{2,2}}=1,\;\sum_{G}\left(R_{\alpha}\ast\mid u_{n}\mid^{p}\right)\mid u_{n}\mid^{p}\to K_{3},n\to\infty.\label{eq:max seq-3}
\end{equation}
Since $p>\frac{N+\alpha}{N-4}$ yields $\frac{2Np}{N+\alpha}\eqqcolon r^{\ast\ast}>2^{\ast\ast}$,
that is $r=\frac{2Np}{N+\alpha+4p}\in\left(\frac{Np}{N+2p},\frac{2Np}{N+4p}\right)$,
we can choose $q$ such that $2<q<r<\infty$. By the interpolation
inequality we have 
\[
C_{\alpha,p,q}\left(\sum(R_{\alpha}\ast\mid u_{n}\mid^{p})\mid u_{n}\mid^{p}\right)^{\frac{q}{2p}}\leq\lVert\Delta u_{n}\rVert_{q}^{q}\leq\lVert\Delta u_{n}\rVert_{2}^{2}\lVert\Delta u_{n}\rVert_{\infty}^{q-2}\leq\lVert\Delta u_{n}\rVert_{\infty}^{q-2},
\]
where $C_{\alpha,p,q}$ is the constant in the inequality (\ref{eq: main inequality-3}).

Taking the limit, we obtain 
\[
C_{\alpha,p,q}K_{3}^{\frac{q}{2p}}\leq\varliminf\limits _{n\to\infty}\lVert\Delta u_{n}\rVert_{\infty}^{q-2}.
\]
This proves $\varliminf\limits _{n\to\infty}\lVert\Delta u_{n}\rVert_{\ell^{\infty}}>0$.

By the translation and taking a subsequence if necessary as Remark
\ref{rem:translation sequence}, we can get a maximizing sequence
$\left\{ u_{n}\right\} $ satisfying (\ref{eq:max seq-3}), $u_{n}\to u$
pointwise on $G$, and $\lvert\Delta u(e)\rvert>0$.

By Lemma \ref{lem: Br=0000E9zis-Lieb lemma-3}, Lemma \ref{lem: Brezis=002013Lieb lemma-2},
the inequalities (\ref{eq:trivial inequality}) and (\ref{eq: main inequality-3}),
then passing to a subsequence if necessary, we have 
\begin{align}
K & _{3}=\underset{n\rightarrow\infty}{\lim}\sum(R_{\alpha}\ast\mid u_{n}\mid^{p})\mid u_{n}\mid^{p}\nonumber \\
 & =\underset{n\rightarrow\infty}{\lim}\frac{\sum(R_{\alpha}\ast\mid u_{n}\mid^{p})\mid u_{n}\mid^{p}}{\lVert u_{n}\rVert_{D^{2,2}}^{2p}}\nonumber \\
 & =\underset{n\rightarrow\infty}{\overline{\lim}}\frac{\sum(R_{\alpha}\ast\mid u_{n}-u\mid^{p})\mid u_{n}-u\mid^{p}+\sum(R_{\alpha}\ast\mid u\mid^{p})\mid u\mid^{p}}{\left(\lVert u_{n}-u\rVert_{D^{2,2}}^{2}+\lVert u\rVert_{D^{2,2}}^{2}\right)^{p}}\label{eq: proof2-1-1}\\
 & \leq\underset{n\rightarrow\infty}{\overline{\lim}}\frac{\sum(R_{\alpha}\ast\mid u_{n}-u\mid^{p})\mid u_{n}-u\mid^{p}+\sum(R_{\alpha}\ast\mid u\mid^{p})\mid u\mid^{p}}{\lVert u_{n}-u\rVert_{D^{2,2}}^{2p}+\lVert u\rVert_{D^{2,2}}^{2p}}\nonumber \\
 & \leq\underset{n\rightarrow\infty}{\overline{\lim}}\frac{K_{3}\lVert u_{n}-u\rVert_{D^{2,2}}^{2p}+\sum(R_{\alpha}\ast\mid u\mid^{p})\mid u\mid^{p}}{\lVert u_{n}-u\rVert_{D^{2,2}}^{2p}+\lVert u\rVert_{D^{2,2}}^{2p}}.\nonumber 
\end{align}
Since $\lVert u\rVert_{D^{2,2}}\neq0$, we have that 
\[
\sum(R_{\alpha}\ast\mid u\mid^{p})\mid u\mid^{p}\geq K_{3}\lVert u\rVert_{D^{2,2}}^{2p},
\]
which yields
\[
\sum(R_{\alpha}\ast\mid u\mid^{p})\mid u\mid^{p}=K_{3}\lVert u\rVert_{D^{2,2}}^{2p}.
\]
By the same argument, $\lVert u\rVert_{D^{2,2}}=1$ and $u$ is a
maximizer. Let $v$ be a solution of
\[
-\Delta v=\mid-\Delta u\mid\in\ell^{2}(G).
\]
Since the Laplace operator $\Delta$ is an isometry from $D^{2,2}(G)$
to $\ell^{2}(G)$ (see \cite[Theorem 12]{HLM22} for a proof), and
by the inequality (\ref{eq: main inequality-3}), 
\[
\sum(R_{\alpha}\ast\mid v\mid^{p})\mid v\mid^{p}\lesssim\lVert\Delta v\rVert_{2}^{2p}=\lVert u\rVert_{D^{2,2}}^{2p}.
\]
Note that $u\leq v$ by the maximum principle. Replacing $u$ by $-u$,
we get $-u\leq v$ similarly. Hence, 
\[
0\leq\mid u\mid\leq v.
\]
In particular we have $\lVert\Delta v\rVert_{2}=\lVert\Delta u\rVert_{2}=1$
and 
\[
\sum(R_{\alpha}\ast\mid u\mid^{p})\mid u\mid^{p}\leq\sum(R_{\alpha}\ast v{}^{p})v{}^{p}.
\]
Therefore, we know that $v$ is a non-negative maximizer. It follows
from the Lagrange multiplier that $v$ is a non-negative solution
of (\ref{eq: bi-harmonic}). The maximum principle yields that it
is positive. 
\end{proof}
For $N\geq5$, $\alpha\in(0,N-4)$, $1<p<\frac{N-\alpha}{4}$, then
$\frac{2Np}{N+\alpha}>p^{\ast\ast}=\frac{Np}{N-2p}$, hence by the
second-order Sobolev inequality (\ref{eq: second-order sobo}) and
the HLS inequality (\ref{eq: hls1}) we get 
\begin{align}
\sum_{G}\left(R_{\alpha}\ast\mid u\mid^{p}\right) & \mid u\mid^{p}\leq C_{p,\alpha}\left(\sum_{G}\mid u\mid^{\frac{2Np}{N+\alpha}}\right)^{\frac{N+\alpha}{N}}\label{eq: main inequality-4}\\
 & \leq C_{p,\alpha}\tilde{C}_{\frac{2Np}{N+\alpha},p}\left(\sum_{G}\mid\Delta u\mid^{p}\right)^{2},\;\forall u\in D^{2,p}(G).\nonumber 
\end{align}

Similarly, we can define the variational problem for (\ref{eq: main inequality-4})
and get a positive solution of (\ref{eq: p-bi-harmonic}) as before.

$\mathbf{\boldsymbol{\mathbf{Acknowledgements}\mathbf{}\text{:}}}$
The author would like to thank Bobo Hua, Xueping Huang, Tao Zhang,
Fengwen Han and Tao Wang for helpful discussions and suggestions. 

\bibliographystyle{plain}
\bibliography{choquard_ref}

\end{document}